\documentclass[12pt]{article}

\usepackage{amsmath}
\usepackage{amssymb}
\usepackage{amsthm}
\usepackage{bbm}
\usepackage{float}
\usepackage{authblk}
\usepackage{mathrsfs}
\usepackage{tikz-cd}
\usepackage{hyperref}
\usepackage{cancel}

\newtheorem{thm}{Theorem}[section]

\newtheorem{prop}[thm]{Proposition}
\newtheorem{lem}[thm]{Lemma}

\theoremstyle{remark}
\newtheorem{rem}[thm]{Remark}

\newcommand{\GL}{\mathrm{GL}}

\title{Proof of Kitaev determinant trivialization conjecture}

\author{Guo Chuan Thiang}
\affil{Beijing International Center for Mathematical Research, Peking University, China}

\date{\today}

\begin{document}
\maketitle

\begin{abstract}
Using ideas from algebraic $K$-theory, we prove that a simple and naturally applicable criterion of Kitaev suffices to trivialize the Fredholm determinant of a multiplicative commutator.
\end{abstract}

\section{Introduction}\label{sec:introduction}
We write $\mathcal{L}\equiv \mathcal{L}(\mathcal{H})$ for the ring of bounded linear operators on a Hilbert space $\mathcal{H}$, and $\mathcal{L}^1\equiv \mathcal{L}^1(\mathcal{H})$ for the ideal of \emph{trace class} operators. The group of invertibles in $\mathcal{L}$ is denoted $\mathcal{L}^\times$, and $\mathbf{1}$ denotes the identity operator on $\mathcal{H}$.

The following Theorem is our main result.
\begin{thm}\label{thm:Kitaev.vanishing}
If $U,V\in\mathcal{L}^\times$ satisfy
\begin{equation}\label{eqn:Kitaev.class}
(U-\mathbf{1})(V-\mathbf{1})\in\mathcal{L}^1\quad\mathrm{and}\quad (V-\mathbf{1})(U-\mathbf{1})\quad \in\mathcal{L}^1,
\end{equation}
then $\det(UVU^{-1}V^{-1})=1$.
\end{thm}
This result was first suggested by A.~Kitaev in his influential paper \cite[pp.~56]{Kitaev}, for the purpose of demonstrating the integrality of certain trace formulae appearing in modern quantum physics (see Section \ref{sec:quantized.traces}). We shall therefore refer to \eqref{eqn:Kitaev.class} as the \emph{Kitaev condition} on an invertible pair $U,V\in \mathcal{L}^\times$. 

\medskip
Recall that an operator $T\in\mathcal{L}$ is \emph{trace class} if for an(y) orthonormal basis $\mathcal{B}$, the sum $\sum_{e\in \mathcal{B}}\big\langle e\big||T|e\big\rangle$ is finite. In this case, its \emph{trace}
\[
\mathrm{Tr}(T)=\sum_{e\in \mathcal{B}}\big\langle e\big|Te\big\rangle \in\mathbb{C}
\]
is well-defined, independently of the choice of $\mathcal{B}$. An operator $U\in\mathcal{L}$ is \emph{determinant class} if $U-\mathbf{1}\in \mathcal{L}^1$, in which case, its \emph{(Fredholm) determinant} $\det(U)\in\mathbb{C}$ is well-defined. See \cite[Chapter 3]{Simon} for a detailed treatment. The identity
\[
\det(\exp(T))=\exp(\mathrm{Tr}(T)),\qquad T\in\mathcal{L}^1,
\]
shows that a trivial determinant is intimately related to integrality of a corresponding trace (up to a $2\pi i$ factor).

The Fredholm determinant is conjugation invariant, and is a group homomorphism from the invertible determinant class operators to the non-zero complex numbers. From these algebraic properties alone, it is straightforward to show that $\det(UVU^{-1}V^{-1})=1$ whenever any of $U,V$ or $UV$ is determinant class. Nevertheless, in general, $\det(UVU^{-1}V^{-1})$ can take \emph{any} non-zero value \cite{AV,EF}, and it is closely related to the rich subject of traces of commutators \cite{HH}, as well as $K$-theory \cite{Brown}.
Now, Kitaev's condition \eqref{eqn:Kitaev.class} implies
\begin{equation}\label{eqn:Kitaev.implies.det.class}
\begin{aligned}
UVU^{-1}V^{-1}-\mathbf{1}&=(UV-VU)U^{-1}V^{-1}\\
&=\big((U-\mathbf{1})(V-\mathbf{1})-(V-\mathbf{1})(U-\mathbf{1})\big)U^{-1}V^{-1}\in\mathcal{L}^1,
\end{aligned}
\end{equation}
so that $\det(UVU^{-1}V^{-1})$ is automatically well-defined, but it is not at all obvious that it must be trivial. Furthermore, because Kitaev's condition occurs naturally in geometric and physics problems (see Section \ref{sec:quantized.traces}), the veracity of Theorem \ref{thm:Kitaev.vanishing} is of considerable general interest for applications. 

More than fifteen years after Kitaev's paper \cite{Kitaev} appeared, some progress was made in \cite{EF}, where Theorem \ref{thm:Kitaev.vanishing} was proved under the extra assumption that $(U^*-\mathbf{1})(V-\mathbf{1})$ and $(V-\mathbf{1})(U^*-\mathbf{1})$ are also trace class. Borel functional calculus played a key role in \cite{EF}, and it is unclear whether the techniques there can be improved to address the full \emph{adjoint-free} Theorem \ref{thm:Kitaev.vanishing}. Nevertheless, the result of \cite{EF} shows that there are no easy counterexamples to Theorem \ref{thm:Kitaev.vanishing}.

Our approach to Theorem \ref{thm:Kitaev.vanishing} is completely different, and we believe that it gets to the heart of the problem. Our proof is presented in Section \ref{sec:proof}, and comprises two main steps, both of which use the Kitaev condition \eqref{eqn:Kitaev.class} in crucial ways.
\begin{enumerate}
\item An \emph{algebraic} step, Prop.~\ref{prop:algebraic.step}, simplifies $\det(UVU^{-1}V^{-1})$ to $\det(E_1E_2E_1^{-1}E_2^{-1})$, where $E_1,E_2$ are so-called elementary matrices. This is the most novel part of the proof. Experts in algebraic $K$-theory may notice that the abstract concept of Steinberg symbols in $K_2$-theory is guiding our manipulations. We choose to keep the presentation concrete, and defer the $K$-theoretic interpretation to Section \ref{sec:K.perspective}.
\item An \emph{analytic} step computes $\det(E_1E_2E_1^{-1}E_2^{-1})=1$ using a formula of Pincus--Helton--Howe and Lidskii's trace theorem.
\end{enumerate}
As the algebraic step involves a longer technical computation, we will first present the short analytic step, assuming the validity of the algebraic step, then complete the proof by justifying the latter assumption.

Finally, the physical context for Theorem \ref{thm:Kitaev.vanishing} is discussed in Section \ref{sec:quantized.traces}.

\section{Proof of determinant trivialization theorem}\label{sec:proof}
For each $n\geq 1$, we write $M_n(-)$ for the $n\times n$ matrix ring with entries in $(-)$, and $\GL_n(\mathcal{L})$ for the group of invertibles in $M_n(\mathcal{L})$. We may identify $M_n(\mathcal{L})$ with $\mathcal{L}(\mathcal{H}^{\oplus n})$, and $M_n(\mathcal{L}^1)$ with the ideal $\mathcal{L}^1(\mathcal{H}^{\oplus n})$. Thus, there are Fredholm determinant homomorphisms
\[
{\det}^{(n)}:\{U\in \GL_n(\mathcal{L})\,:\,U-\mathbf{1}^{\oplus n}\in M_n(\mathcal{L}^1)\}\to\mathbb{C}^\times,\qquad n\geq 1,
\]
and they are compatible with $\det=\det^{(1)}$ in the sense that
\begin{equation}\label{eqn.det.stable}
{\det}^{(n)}\begin{pmatrix} U & 0 \\ 0 & \mathbf{1}^{\oplus (n-1)}\end{pmatrix} = {\det}^{(1)}\, U,\qquad n\geq 1.
\end{equation}
By embedding in $n\times n$ matrices, the ``hidden'' invariances of ${\det}^{(n)}$ under conjugation by $\GL_n(\mathcal{L})$ (not just $\GL_1(\mathcal{L})$) becomes available for exploitation.

For $1\leq i\neq j\leq n$, and $T\in\mathcal{L}$, the \emph{elementary matrix} $e_{ij}(T)$ is the $n\times n$ matrix whose $ij$-entry is $T$, diagonal entries are $\mathbf{1}$, and remaining entries are $0$. For example,
\[
e_{21}(T)=\begin{pmatrix} \mathbf{1} & 0 \\ T & \mathbf{1}\end{pmatrix},\qquad e_{13}(T)=\begin{pmatrix} \mathbf{1} & 0 & T \\ 0 & \mathbf{1} & 0 \\ 0 & 0 & \mathbf{1}\end{pmatrix}.
\]
An elementary matrix $e_{ij}(T)$ is determinant class iff $T\in\mathcal{L}^1$, in which case, $\det(e_{ij}(T))=1$.

\begin{proof}[Proof of Theorem \ref{thm:Kitaev.vanishing}]
We will use the following identity
\begin{equation}\label{eqn:algebraic.simplification}
\det(UVU^{-1}V^{-1})={\det}^{(2)}\big(e_{12}(U-\mathbf{1})e_{21}(V-\mathbf{1})e_{12}(\mathbf{1}-U)e_{21}(\mathbf{1}-V)\big),
\end{equation}
valid whenever $U,V\in\mathcal{L}^\times$ satisfy the Kitaev condition \eqref{eqn:Kitaev.class}. The derivation of \eqref{eqn:algebraic.simplification} will be provided in Prop.~\ref{prop:algebraic.step} later.

The elementary matrices appearing on the right side of \eqref{eqn:algebraic.simplification} are exponentials,
\begin{align*}
e_{12}(U-\mathbf{1})&=\begin{pmatrix} \mathbf{1} & 0 \\ 0 & \mathbf{1} \end{pmatrix}+\begin{pmatrix} 0 & U-\mathbf{1} \\ 0 & 0 \end{pmatrix}=e^A, & A:= \begin{pmatrix} 0 & U-\mathbf{1} \\ 0 & 0 \end{pmatrix},\\
e_{21}(V-\mathbf{1})&=\begin{pmatrix} \mathbf{1} & 0 \\ 0 & \mathbf{1} \end{pmatrix}+\begin{pmatrix} 0 & 0 \\ V-\mathbf{1} & 0 \end{pmatrix}=e^B, & B:= \begin{pmatrix} 0 & 0 \\ V-\mathbf{1} & 0 \end{pmatrix}.
\end{align*}
Then \eqref{eqn:algebraic.simplification} becomes
\begin{equation}\label{eqn:algebraic.simplification.and.HH}
\det(UVU^{-1}V^{-1})={\det}^{(2)}\big(e^Ae^Be^{-A}e^{-B}\big)=\exp(\mathrm{Tr}[A,B]),
\end{equation}
where the second equality is the Pincus--Helton--Howe formula (\cite{PincusThesis}, \cite[pp.~182]{HH}), valid provided $[A,B]$ is trace class. The latter trace class membership holds because of the Kitaev condition \eqref{eqn:Kitaev.class} and
\[
AB=\begin{pmatrix} (U-\mathbf{1})(V-\mathbf{1}) & 0 \\ 0 & 0 \end{pmatrix},\qquad BA=\begin{pmatrix} 0 & 0 \\ 0 & (V-\mathbf{1})(U-\mathbf{1}) \end{pmatrix}.
\]
Furthermore, by Lidskii's trace theorem (\cite{Lidskii}, \cite[Cor.~3.8]{Simon}),
\[
\mathrm{Tr}[A,B]=\mathrm{Tr}(AB)-\mathrm{Tr}(BA)=0,
\]
Putting this into \eqref{eqn:algebraic.simplification.and.HH} gives $\det(UVU^{-1}V^{-1})=1$.
\end{proof}

\medskip

The proof of the identity \eqref{eqn:algebraic.simplification} will involve elementary matrix manipulation. Recall that elementary matrices satisfy the following \emph{Steinberg relations}:
\begin{itemize}
\item {[Commutation.]} If the ``inner'' indices are distinct, and the ``outer'' indices are distinct, then
\begin{equation}
e_{ij}(S)e_{l k}(T)= e_{l k}(T)e_{ij}(S),\qquad i\neq k,\;j\neq l. 
\label{eqn:Steinberg.commute.no.inner.outer}
\end{equation}
In the special case $l=i\neq j=k$, a further equality holds,
\begin{equation}\label{eqn:Steinberg.merger}
e_{ij}(S)e_{ij}(T)=e_{ij}(S+T)=e_{ij}(T)e_{ij}(S).
\end{equation}
Note that the inversion formula $e_{ij}(T)^{-1}=e_{ij}(-T)$ follows from \eqref{eqn:Steinberg.merger}.
\item {[Commutator.]} If the ``inner'' indices coincide, \emph{or} the ``outer'' indices coincide (but not both),
\begin{equation}\label{eqn:Steinberg.commutator.123}
\begin{aligned}
e_{ij}(S)e_{jk}(T)&=e_{ik}(ST)\cdot e_{jk}(T)e_{ij}(S),\\
e_{ji}(S)e_{kj}(T)&=e_{ki}(-TS)\cdot e_{kj}(T)e_{ji}(S),\qquad i\neq k.
\end{aligned}
\end{equation}
\end{itemize}

To analyse $\det(UVU^{-1}V^{-1})$ in terms of elementary matrices, we first embed $UVU^{-1}V^{-1}$ into $3\times 3$ matrices to obtain the following identity,
\begin{equation}\label{eqn:mult.commutator.stabilized}
\begin{aligned}
&\begin{pmatrix}UVU^{-1}V^{-1} & 0 & 0 \\ 0 & \mathbf{1} & 0  \\ 0 & 0 & \mathbf{1}\end{pmatrix}
\\
&=
\underbrace{\begin{pmatrix}U & 0 & 0 \\ 0 & \mathbf{1} & 0  \\ 0 & 0 & U^{-1}\end{pmatrix}}_{=:d_{13}(U)}
\underbrace{\begin{pmatrix}V & 0 & 0 \\ 0 & V^{-1} & 0  \\ 0 & 0 & \mathbf{1}\end{pmatrix}}_{=:d_{12}(V)}
\begin{pmatrix}U^{-1} & 0 & 0 \\ 0 & \mathbf{1} & 0  \\ 0 & 0 & U\end{pmatrix}
\begin{pmatrix}V^{-1} & 0 & 0 \\ 0 & V & 0  \\ 0 & 0 & \mathbf{1}\end{pmatrix}.
\end{aligned}
\end{equation}
By a famous observation of Whitehead,
\begin{equation*}
\begin{pmatrix}V & 0 \\ 0 & V^{-1}\end{pmatrix}
=\begin{pmatrix}\mathbf{1} & 0 \\ V^{-1} & \mathbf{1}\end{pmatrix}
\begin{pmatrix}\mathbf{1} & \mathbf{1}-V \\ 0 & \mathbf{1}\end{pmatrix}
\begin{pmatrix}\mathbf{1} & 0 \\ -\mathbf{1} & \mathbf{1}\end{pmatrix}
\begin{pmatrix}\mathbf{1} & \mathbf{1}-V^{-1} \\ 0 & \mathbf{1}\end{pmatrix},
\end{equation*}
so $d_{12}(V)$ is a product of four elementary matrices; similarly for $d_{13}(U)$. Therefore,
\begin{equation}\label{eqn:factorize.to.d.matrices}
\begin{pmatrix}UVU^{-1}V^{-1} & 0 & 0 \\ 0 & \mathbf{1} & 0  \\ 0 & 0 & \mathbf{1}\end{pmatrix}=d_{13}(U)d_{12}(V)d_{13}(U)^{-1}d_{12}(V)^{-1}
\end{equation}
is a product of sixteen elementary matrices (four groups of four), since
\begin{equation}\label{eqn:for.shuffling}
\begin{aligned}
d_{13}(U) & = e_{31}(U^{-1})e_{13}(\mathbf{1}-U)e_{31}(-\mathbf{1})e_{13}(\mathbf{1}-U^{-1}),\\
d_{12}(V) & = e_{21}(V^{-1})e_{12}(\mathbf{1}-V)e_{21}(-\mathbf{1})e_{12}(\mathbf{1}-V^{-1}).
\end{aligned}
\end{equation}

We would like to swap the positions of $d_{13}(U)$ and $d_{12}(V)$ in \eqref{eqn:factorize.to.d.matrices}, in order to cancel $d_{13}(U)^{-1}d_{12}(V)^{-1}$. This entails moving the four elementary matrix factors in $d_{13}(U)$ past those of $d_{12}(V)$. Because of the Steinberg relation \eqref{eqn:Steinberg.commutator.123}, we will pick up a complicated collection of commutator terms of the form $e_{ik}(ST)$ or $e_{ki}(-TS)$. Fortunately, many of these terms will have trivial determinant, so they are ``redundant'' for the purposes of computing the determinant, because of the following easy lemma.
\begin{lem}\label{lem:redundant.elementary}
Let $W,W_0,W^\prime\in \GL_n(\mathcal{L})$ have determinant class product $WW_0W^\prime$. If $W_0$ is determinant class with trivial determinant, then
\[
{\det}^{(n)}(WW_0W^\prime)={\det}^{(n)}(WW^\prime).
\]
\end{lem}
\begin{proof}
By conjugation invariance, $W_0W^\prime W$ is determinant class, and has the same determinant as $WW_0W^\prime$. Under the extra assumption that ${\det}^{(n)}(W_0)=1$ is well-defined, $W^\prime W=W_0^{-1}(W_0W^\prime W)$ is determinant class. By multiplicativity and conjugation invariance again, we have
\[
{\det}^{(n)}(WW_0W^\prime)={\det}^{(n)}(W_0W^\prime W)=\cancel{{\det}^{(n)}(W_0)}\cdot{\det}^{(n)}(W^\prime W)={\det}^{(n)}(WW^\prime).
\]
\end{proof}

\begin{prop}\label{prop:algebraic.step}
Let $U,V\in\mathcal{L}^\times$ satisfy the Kitaev condition \eqref{eqn:Kitaev.class}. Then the determinant identity \eqref{eqn:algebraic.simplification} holds.
\end{prop}
\begin{proof}
The Kitaev condition \eqref{eqn:Kitaev.class} ensures that 
$e_{ij}((U-\mathbf{1})(V-\mathbf{1}))$ is determinant class, with trivial determinant. Then, for $i\neq k$,
\begin{equation}\label{eqn:extra.commutativity}
\begin{aligned}
&{\det}^{(3)}\big(W e_{ij}(U-\mathbf{1})e_{jk}(V-\mathbf{1}) W^\prime\big)\\
&={\det}^{(3)}\big(W \underbrace{e_{ik}((U-\mathbf{1})(V-\mathbf{1}))}_{\text{redundant by Lemma}\;\ref{lem:redundant.elementary}}e_{jk}(V-\mathbf{1})e_{ij}(U-\mathbf{1}) W^\prime\big) & (\mathrm{Eq.}\,\eqref{eqn:Steinberg.commutator.123})\\
&={\det}^{(3)}\big(W e_{jk}(V-\mathbf{1})e_{ij}(U-\mathbf{1}) W^\prime\big)
\end{aligned}
\end{equation}
holds whenever $W,W^\prime$ are invertible operators such that the above determinants are defined. Similarly with $U$ and $V$ swapped. 
It will be convenient to write $C \simeq D$ to indicate that $\det(WCW^\prime)=\det(WDW^\prime)$ holds whenever $W,W^\prime$ are invertibles such that $WCW^\prime, WDW^\prime$ are determinant class. With this notation, Eq.~\eqref{eqn:extra.commutativity} becomes the extra ``equality'' \eqref{eqn:Steinberg.Kitaev.extra.commute} below,
\begin{align}
e_{ij}(U-\mathbf{1})e_{jk}(V-\mathbf{1}) &\simeq e_{jk}(V-\mathbf{1})e_{ij}(U-\mathbf{1}), & i\neq k,\label{eqn:Steinberg.Kitaev.extra.commute}
\\
e_{ij}((U-\mathbf{1})V)&\simeq e_{ij}(U-\mathbf{1})\simeq e_{ij}(V(U-\mathbf{1})),\label{eqn:Steinberg.Kitaev.extra.identity.2}
\end{align}
while the extra ``equality'' \eqref{eqn:Steinberg.Kitaev.extra.identity.2} is due to 
\[
(U-\mathbf{1})V=U-\mathbf{1}=V(U-\mathbf{1})
\qquad \mathrm{mod}\;\mathcal{L}^1,\qquad (\text{by Kitaev condition}\;\eqref{eqn:Kitaev.class})
\]
together with \eqref{eqn:Steinberg.merger} and the redundancy Lemma \ref{lem:redundant.elementary}.
Since Kitaev's condition holds with $U,V$ swapped, and/or $U$ replaced by $U^{-1}$ and/or $V$ replaced by $V^{-1}$, the same replacements can be made in the ``equalities'' \eqref{eqn:Steinberg.Kitaev.extra.commute}--\eqref{eqn:Steinberg.Kitaev.extra.identity.2}.

Recall the $3\times 3$ matrix identity \eqref{eqn:factorize.to.d.matrices}, which implies that
\begin{equation}\label{eqn:det.alg.calc.first.factorization}
\det(UVU^{-1}V^{-1})={\det}^{(3)}\big(d_{13}(U)d_{12}(V)d_{13}(U)^{-1}d_{12}(V)^{-1}\big).
\end{equation}
Here, $d_{13}(U)$ and $d_{12}(V)$ are the elementary matrix products given in \eqref{eqn:for.shuffling}.
We will swap $d_{13}(U)$ with $d_{12}(V)$, following the Steinberg relations \eqref{eqn:Steinberg.commute.no.inner.outer}, \eqref{eqn:Steinberg.merger}, \eqref{eqn:Steinberg.commutator.123}, and the extra $\simeq$ ``equalities'' \eqref{eqn:Steinberg.Kitaev.extra.commute}, \eqref{eqn:Steinberg.Kitaev.extra.identity.2}. 
The first step is to move the factor $\boxed{e_{13}(\mathbf{1}-U^{-1})}$ in $d_{13}(U) = \boxed{e_{31}(U^{-1})} \boxed{e_{13}(\mathbf{1}-U)} \boxed{e_{31}(-\mathbf{1})} \boxed{e_{13}(\mathbf{1}-U^{-1})}$
to the right of $d_{12}(V)$. To aid the reader, we underline the extra commutator terms that arise during this step,
\begin{align*}
&d_{13}(U)d_{12}(V)d_{13}(U)^{-1}d_{12}(V)^{-1}\\
&=\boxed{e_{31}(U^{-1})} \boxed{e_{13}(\mathbf{1}-U)} \boxed{e_{31}(-\mathbf{1})} \boxed{e_{13}(\mathbf{1}-U^{-1})}  \\
&\qquad e_{21}(V^{-1}) e_{12}(\mathbf{1}-V) e_{21}(-\mathbf{1}) e_{12}(\mathbf{1}-V^{-1}) \cdot d_{13}(U)^{-1} d_{12}(V)^{-1}\\
&\simeq \boxed{e_{31}(U^{-1})} \boxed{e_{13}(\mathbf{1}-U)} \boxed{e_{31}(-\mathbf{1})} \underline{ e_{23}\big(-V^{-1}(\mathbf{1}-U^{-1})\big)}  \\
&\qquad e_{21}(V^{-1}) \boxed{e_{13}(\mathbf{1}-U^{-1})} e_{12}(\mathbf{1}-V) e_{21}(-\mathbf{1}) e_{12}(\mathbf{1}-V^{-1})\\
&\qquad \cdot d_{13}(U)^{-1}d_{12}(V)^{-1} & ({\rm by}\;\eqref{eqn:Steinberg.commutator.123})\\
&\simeq \boxed{e_{31}(U^{-1})} \boxed{e_{13}(\mathbf{1}-U)} \boxed{e_{31}(-\mathbf{1})} \underline{ e_{23}(U^{-1}-\mathbf{1})}  \\
&\qquad e_{21}(V^{-1}) e_{12}(\mathbf{1}-V) \boxed{e_{13}(\mathbf{1}-U^{-1})} e_{21}(-\mathbf{1}) e_{12}(\mathbf{1}-V^{-1})\\
&\qquad \cdot d_{13}(U)^{-1}d_{12}(V)^{-1} & ({\rm by}\;\eqref{eqn:Steinberg.Kitaev.extra.identity.2}, \eqref{eqn:Steinberg.commute.no.inner.outer})\\
&\simeq \boxed{e_{31}(U^{-1})} \boxed{e_{13}(\mathbf{1}-U)} \boxed{e_{31}(-\mathbf{1})} \cancel{\underline{ e_{23}(U^{-1}-\mathbf{1})}}  \\
&\qquad e_{21}(V^{-1}) e_{12}(\mathbf{1}-V) \cancel{\underline{e_{23}(\mathbf{1}-U^{-1})}} e_{21}(-\mathbf{1}) e_{12}(\mathbf{1}-V^{-1}) \\
&\qquad \boxed{e_{13}(\mathbf{1}-U^{-1})} \cdot d_{13}(U)^{-1}d_{12}(V)^{-1}. & ({\rm by}\;\eqref{eqn:Steinberg.commutator.123}, \eqref{eqn:Steinberg.commute.no.inner.outer})
\end{align*}
The last cancellation holds because $e_{23}(U^{-1}-\mathbf{1})$ can be commuted past $e_{21}(-\mathbf{1}) e_{12}(\mathbf{1}-V^{-1})$ due to \eqref{eqn:Steinberg.commute.no.inner.outer} and \eqref{eqn:Steinberg.Kitaev.extra.commute}, whence it cancels with $e_{23}(\mathbf{1}-U^{-1})$.

We shall be brief about the next three steps, where the remaining three factors $\boxed{e_{31}(-\mathbf{1})}$, $\boxed{e_{13}(\mathbf{1}-U)}$, $\boxed{e_{31}(U^{-1})}$ of $d_{13}(U)$ are successively moved to the right,
\begin{align*}
&d_{13}(U)d_{12}(V)d_{13}(U)^{-1}d_{12}(V)^{-1}\\
&\simeq \boxed{e_{31}(U^{-1})} \boxed{e_{13}(\mathbf{1}-U)} e_{21}(V^{-1}) \underline{e_{32}(V-\mathbf{1})} e_{12}(\mathbf{1}-V) e_{21}(-\mathbf{1})\\
&\qquad  \underline{e_{32}(V^{-1}-\mathbf{1})} e_{12}(\mathbf{1}-V^{-1}) \boxed{e_{31}(-\mathbf{1})} \boxed{e_{13}(\mathbf{1}-U^{-1})}\cdot d_{13}(U)^{-1}d_{12}(V)^{-1}\\
&\simeq \boxed{e_{31}(U^{-1})} \underline{e_{23}(U-\mathbf{1})} e_{21}(V^{-1}) \underline{e_{32}(V-\mathbf{1})} e_{12}(\mathbf{1}-V) \underline{e_{23}(\mathbf{1}-U)} e_{21}(-\mathbf{1})\\
&\qquad  \underline{e_{32}(V^{-1}-\mathbf{1})} e_{12}(\mathbf{1}-V^{-1}) \boxed{e_{13}(\mathbf{1}-U)} \boxed{e_{31}(-\mathbf{1})} \boxed{e_{13}(\mathbf{1}-U^{-1})}\\
&\qquad\cdot d_{13}(U)^{-1}d_{12}(V)^{-1}\\
&\simeq \underline{e_{21}(U^{-1}-\mathbf{1})} \underline{e_{23}(U-\mathbf{1})} e_{21}(V^{-1}) \cancel{\underline{e_{32}(V-\mathbf{1})} \underline{e_{32}(\mathbf{1}-V)}} e_{12}(\mathbf{1}-V) \\
&\qquad  \underline{e_{21}(\mathbf{1}-U^{-1})} \underline{e_{23}(\mathbf{1}-U)} e_{21}(-\mathbf{1})\cancel{\underline{e_{32}(V^{-1}-\mathbf{1})} \underline{e_{32}(\mathbf{1}-V^{-1})}} e_{12}(\mathbf{1}-V^{-1}) \\
&\qquad\underbrace{\boxed{e_{31}(U^{-1})} \boxed{e_{13}(\mathbf{1}-U)} \boxed{e_{31}(-\mathbf{1})} \boxed{e_{13}(\mathbf{1}-U^{-1})}}_{\cancel{d_{13}(U)}}\cdot \cancel{d_{13}(U)^{-1}}d_{12}(V)^{-1}\\
&\simeq \underline{e_{21}(U^{-1}-\mathbf{1})} e_{21}(V^{-1}) e_{12}(\mathbf{1}-V) \\
&\qquad  \underline{e_{21}(\mathbf{1}-U^{-1})} \cancel{\underline{e_{23}(U-\mathbf{1})}\underline{e_{23}(\mathbf{1}-U)}} e_{21}(-\mathbf{1}) e_{12}(\mathbf{1}-V^{-1})\cdot d_{12}(V)^{-1}\\
&= e_{21}(V^{-1})\underline{e_{21}(U^{-1}-\mathbf{1})}e_{12}(\mathbf{1}-V) \underline{e_{21}(\mathbf{1}-U^{-1})}e_{12}(V-\mathbf{1})e_{21}(-V^{-1}),
\end{align*}
where in the last equality, we substituted \eqref{eqn:for.shuffling} for $d_{12}(V)$.
Taking determinants, invoking \eqref{eqn:det.alg.calc.first.factorization}, and using cyclicity, we arrive at
\begin{equation}\label{eqn:algebraic.prop.last.moves}
\det(UVU^{-1}V^{-1})={\det}^{(2)}\big(e_{12}(V-\mathbf{1}) e_{21}(U^{-1}-\mathbf{1}) e_{12}(\mathbf{1}-V) e_{21}(\mathbf{1}-U^{-1}) \big).
\end{equation}
Swapping $(U,V)$ with $(V^{-1},U)$, we also have
\[
\det(V^{-1}UVU^{-1})={\det}^{(2)}\big(e_{12}(U-\mathbf{1}) e_{21}(V-\mathbf{1}) e_{12}(\mathbf{1}-U) e_{21}(\mathbf{1}-V)\big).
\]
Since $\det(UVU^{-1}V^{-1})=\det(V^{-1}UVU^{-1})$ by cyclicity, the identity \eqref{eqn:algebraic.simplification} follows.
\end{proof}

\section{Discussion}
\subsection{$K$-theory perspective}\label{sec:K.perspective}
A pedagogical treatment of algebraic $K_2$-theory can be found in \cite[Chapter 4]{Rosenberg}. Below, we explain how the algebraic step in our proof (Prop.~\ref{prop:algebraic.step}) is related to Steinberg symbols in $K_2$-theory. 

For a unital ring $R$ with an ideal $I$, we write $\GL_n(R)$ for the group of invertibles in $M_n(R)$, and $\GL_n(R,I)\subseteq \GL_n(R)$ for the kernel of the induced quotient homomorphism $\GL_n(R)\twoheadrightarrow\GL_n(R/I)$. In the special case
\[
R=\mathcal{L}=\mathcal{L}(\mathcal{H}),\qquad I=\mathcal{L}^1=\mathcal{L}^1(\mathcal{H}), 
\]
$\GL_n(\mathcal{L},\mathcal{L}^1)$ is precisely the group of invertible determinant class operators on $\mathcal{H}^{\oplus n}$, and we have a homomorphism $\det^{(n)}:\GL_n(\mathcal{L},\mathcal{L}^1)\to \mathbb{C}^\times$. By the stability \eqref{eqn.det.stable}, we have a homomorphism
\[
\widetilde{\det}:\GL(\mathcal{L},\mathcal{L}^1):=\bigcup_{n\geq 1}\GL_n(\mathcal{L},\mathcal{L}^1)\to\mathbb{C}^\times.
\]
Because $\widetilde{\det}$ is trivial on elementary matrices of the form $e_{ij}(T), T\in \mathcal{L}^1$, and is invariant under conjugation, it vanishes on the normal subgroup $\mathrm{E}(\mathcal{L},\mathcal{L}^1)$ generated by 
$\{e_{ij}(T)\,:\,T\in \mathcal{L}^1\}$. So $\widetilde{\det}$ descends to a homomorphism on the \emph{relative} $K_1$-group,
\[
\check{\det}:K_1(\mathcal{L},\mathcal{L}^1):=\GL(\mathcal{L},\mathcal{L}^1)/\mathrm{E}(\mathcal{L},\mathcal{L}^1)\to\mathbb{C}^\times.
\]
Thus, the determinant of a determinant class commutator $UVU^{-1}V^{-1}$ depends only on the class $[UVU^{-1}V^{-1}]_{K_1(\mathcal{L},\mathcal{L}^1)}$.

At the same time, it is easy to see that $\det(UVU^{-1}V^{-1})$ depends only on the images $u,v$ of $U,V$ in the quotient ring $\mathcal{L}/\mathcal{L}^1$. Furthermore,
\[
UVU^{-1}V^{-1}-\mathbf{1}=(UV-VU)U^{-1}V^{-1}\in \mathcal{L}^1,
\]
so $UV-VU$ is trace class, thus $u,v$ commute in $\mathcal{L}/\mathcal{L}^1$. Therefore, there should be some invariant $\{u,v\}$ of the \emph{commuting units} $u,v\in \mathcal{L}/\mathcal{L}^1$, which controls $[UVU^{-1}V^{-1}]_{K_1(\mathcal{L},\mathcal{L}^1)}$ and thus $\det(UVU^{-1}V^{-1})$. This invariant is precisely the \emph{Steinberg symbol} of $u,v$, whose general definition is recalled next.

Let $S$ be a unital ring. There is a notion of elementary matrices $e_{ij}(s), s\in S, 1\leq i\neq j\leq \infty$, and they generate a subgroup $\mathrm{E}(S)\subseteq \GL(S)$. The elementary matrices $e_{ij}(s)$ satisfy the general Steinberg relations \eqref{eqn:Steinberg.commute.no.inner.outer}, \eqref{eqn:Steinberg.merger}, \eqref{eqn:Steinberg.commutator.123}, as well as some other $S$-dependent matrix relations. The \emph{Steinberg group} $\mathrm{St}(S)$ is the free group generated by abstract \emph{Steinberg symbols} $x_{ij}(s)$, $s\in S$, $1\leq i\neq j< \infty$, subject only to the Steinberg relations \eqref{eqn:Steinberg.commute.no.inner.outer}, \eqref{eqn:Steinberg.merger}, \eqref{eqn:Steinberg.commutator.123}. So there is a surjective group homomorphism $\varphi_S:\mathrm{St}(S)\to \mathrm{E}(S)$ mapping $x_{ij}(s)\mapsto e_{ij}(s)$. By definition, $K_2(S)=\ker \varphi_S$, and conceptually, $K_2(S)$ encodes the extra ``non-obvious'' relations satisfied by the $e_{ij}(s)$ other than the ``obvious'' Steinberg relations.

For commuting units $u,v\in S$, consider the product $d_{13}(u)d_{12}(v)d_{13}(u)^{-1}d_{12}(v)^{-1}$ of diagonal matrices, as in \eqref{eqn:mult.commutator.stabilized}, and recall that $d_{13}(u)$ and $d_{12}(v)$ are themselves products of elementary matrices by Whitehead's lemma. The \emph{Steinberg symbol} of $u,v$ is defined as
\begin{equation}\label{eqn:Steinberg.symbol.definition}
\{u,v\}=D_{13}(u)D_{12}(v)D_{13}(u)^{-1}D_{12}(v)^{-1}\in K_2(S),
\end{equation}
where $D_{13}(u), D_{12}(v)\in\mathrm{St}(S)$ are respective lifts of $d_{13}(u), d_{12}(v)\in \mathrm{E}(S)$. Here, $\{u,v\}$ lands in $K_2(S)=\ker \varphi_S$, because of commutativity of $u,v$.
For $S=R/I$ a quotient ring, general $K$-theory yields a connecting homomorphism
\begin{align*}
\partial:K_2(R/I)&\to K_1(R,I)\\
x_{i_1j_1}(s_1)\cdots x_{i_mj_m}(s_m) &\mapsto \big[e_{i_1j_1}(\tilde{s}_1)\cdots e_{i_mj_m}(\tilde{s}_m)\big]_{K_1(R,I)},
\end{align*}
where $\tilde{s}_k\in R$ are any choices of lifts of $s_k\in R/I$; see \cite[Chapter 4.3]{Rosenberg} for details.

Now specialize to the case $R=\mathcal{L}$ and $I=\mathcal{L}^1$. For $U,V\in \mathcal{L}^\times$ with determinant class $UVU^{-1}V^{-1}$, the above general definitions give
\begin{equation}\label{eqn:Brown.observation}
\begin{aligned}
\check{\det}\circ\partial\{u,v\}
&=\check{\det}\big[d_{13}(U)d_{12}(V)d_{13}(U)^{-1}d_{12}(V)^{-1}\big]_{K_1(\mathcal{L},\mathcal{L}^1)}\\
&=\check{\det}\left[\begin{pmatrix}UVU^{-1}V^{-1} & 0 & 0 \\ 0 & \mathbf{1} & 0 \\ 0 & 0 & \mathbf{1}\end{pmatrix}\right]_{K_1(\mathcal{L},\mathcal{L}^1)}\\
&=\det(UVU^{-1}V^{-1}).
\end{aligned}
\end{equation}
This $K$-theoretic nature of $\det(UVU^{-1}V^{-1})$ was first observed by Brown \cite{Brown}. 

\medskip
A minor translation of the calculation in our Prop.~\ref{prop:algebraic.step} leads to an analogous abstract Steinberg symbol formula, which may be of independent interest to algebraists.
\begin{thm}\label{thm:abstract.algebraic.step}
Let $u,v$ be units in a unital ring $S$. If
\begin{equation}\label{eqn:abstract.Kitaev.condition}
(u-1)(v-1)=0=(v-1)(u-1),
\end{equation}
then $\{u,v\}$ is a commutator of Steinberg group generators,
\begin{equation}\label{eqn:Steinberg.Kitaev.abstract.simplification}
\{u,v\}=x_{12}(u-1)x_{21}(v-1)x_{12}(1-u)x_{21}(1-v).
\end{equation}
\end{thm}
\begin{proof}
Condition \eqref{eqn:abstract.Kitaev.condition} implies commutativity of $u,v$, so the Steinberg symbol $\{u,v\}\in K_2(S)$ is well-defined. 

We had used Whitehead's lemma to factorize $d_{13}(U)d_{12}(V)d_{13}(U)^{-1}d_{12}(V)^{-1}$ into a product of elementary matrices, see Eq.~\eqref{eqn:for.shuffling}. The Steinberg symbol $\{u,v\}=D_{13}(u)D_{12}(v)D_{13}(u)^{-1}D_{12}(v)^{-1}$, Eq.~\eqref{eqn:Steinberg.symbol.definition}, has precisely the same factorization, but with Steinberg group generators $x_{ij}(-)$ instead of elementary matrices $e_{ij}(-)$, and $u,v$ instead of $U,V$. Besides the Steinberg relations, the $x_{ij}(-)$ enjoy some extra relations due to the condition \eqref{eqn:abstract.Kitaev.condition}. Specifically, these extra relations are \eqref{eqn:Steinberg.Kitaev.extra.commute}--\eqref{eqn:Steinberg.Kitaev.extra.identity.2}, but with genuine \emph{equality} rather than $\simeq$, and $x_{ij}(-)$ instead of $e_{ij}(-)$. 

As in the proof of Prop.~\ref{prop:algebraic.step}, we wish to move the $x_{ij}(-)$ factors in $D_{13}(u)$ past those of $D_{12}(v)$. The calculation proceeds in exactly the same way (but with equality instead of $\simeq$). The analogue of cyclicity of determinant is the centrality of $K_2(S)$ in the Steinberg group $\mathrm{St}(S)$ (\cite[Theorem 4.2.4]{Rosenberg}). Thus, the result \eqref{eqn:algebraic.prop.last.moves} there translates directly to the equality
\[
\{u,v\}=x_{12}(v-1) x_{21}(u^{-1}-1) x_{12}(1-v) x_{21}(1-u^{-1}).
\]
Using the identity $\{u,v\}=\{v^{-1},u\}$, we arrive at the equivalent formula \eqref{eqn:Steinberg.Kitaev.abstract.simplification}.
\end{proof}

\begin{rem}
It is known that $K_1(\mathcal{L},\mathcal{L}^1)\cong \mathcal{V}\oplus \mathbb{C}^\times$, where $\mathcal{V}$ is the abelian group underlying a vector space of uncountable dimension \cite{AV}. So the determinant only detects a small part of $K_1(\mathcal{L},\mathcal{L}^1)$. In view of formula \eqref{eqn:Brown.observation} and Theorem \ref{thm:Kitaev.vanishing}, one might speculate that the Kitaev condition \eqref{eqn:Kitaev.class} already implies triviality of $\{u,v\}\in K_2(\mathcal{L}/\mathcal{L}^1)$, or of $[UVU^{-1}V^{-1}]\in K_1(\mathcal{L},\mathcal{L}^1)$. We defer such questions to future work. 
\end{rem}

\subsection{Application to quantized traces}\label{sec:quantized.traces}\label{sec:application.quantized.traces}
Let $A,B\in\mathcal{L}$, and suppose 
\begin{equation}\label{eqn:quantized.trace.hypothesis}
[A,B]\in \mathcal{L}^1\;\;\mathrm{and}\;\;(e^A-\mathbf{1})(e^B-\mathbf{1})\in\mathcal{L}^1.
\end{equation}
Then, as in \eqref{eqn:Kitaev.implies.det.class}, $(e^B-\mathbf{1})(e^A-\mathbf{1})$ is trace class as well. Thus, $e^A,e^B$ satisfy Kitaev's condition \eqref{eqn:Kitaev.class}, so Theorem \ref{thm:Kitaev.vanishing} guarantees that $\det(e^Ae^Be^{-A}e^{-B})=1$. By the Pincus--Helton--Howe formula \eqref{eqn:algebraic.simplification.and.HH}, we learn that the condition \eqref{eqn:quantized.trace.hypothesis} implies the \emph{trace quantization},
\[
\frac{1}{2\pi i}\mathrm{Tr}[A,B]\in\mathbb{Z}.
\]

The conditions of \eqref{eqn:quantized.trace.hypothesis} are naturally satisfied in the following physical situation. Suppose $P$ is a spectral projection for a compact isolated part of the spectrum of some Schr\"{o}dinger operator on $\mathcal{H}=L^2(\mathbb{R}^2)$. Let $X$ and $Y$ be multiplication operators by the indicator functions on the right half-plane and the upper half-plane, respectively. Under mild assumptions, $P$ is a locally trace class operator with integral kernel which decays rapidly away from the diagonal, leading to $[PXP,PYP]$ being trace class (see \cite{Thiang, LT2023}). Consider
\begin{equation*}
A=2\pi i PXP,\qquad B=2\pi i PYP.
\end{equation*}
The holomorphic function $\phi(z)=\exp(2\pi i z)-1$ factorizes as
\[
\phi(z)=z(1-z)\psi(z)=\psi(z)(1-z)z,
\]
where $\psi$ is some other holomorphic function. Therefore,
\begin{align*}
(e^A-\mathbf{1})(e^B-\mathbf{1})&=\phi(PXP)\phi(PYP)\\
&=\psi(PXP)(PXP(\mathbf{1}-PXP))(PYP(\mathbf{1}-PYP))\psi(PYP)\\
&=\psi(PXP)\underbrace{(PXP(\mathbf{1}-X)P)(PYP(\mathbf{1}-Y)P)}_{\in\mathcal{L}^1}\psi(PYP),
\end{align*}
where the trace class membership in the last line is deduced from the fact that the supports of $X,(\mathbf{1}-X),Y,(\mathbf{1}-Y)$ intersect at a point, and the above-mentioned properties of $P$ (see \cite{LT2023} for details). Therefore, the condition \eqref{eqn:quantized.trace.hypothesis} holds, and it implies that
\[
\frac{1}{2\pi i}\mathrm{Tr}[A,B]=2\pi i\cdot\mathrm{Tr}[PXP,PYP]\in\mathbb{Z}.
\]
Examples of $P$ for which the above integer is nonzero are known --- they arise from coarse index theory \cite{LT2023, LTfinite}, and the exactly quantized Hall conductance in physics \cite{Thiang}.

\section*{Acknowledgements}
We thank X.~Tang for his suggestion to investigate Steinberg symbols, as well as M.~Ludewig and J.~Rosenberg for helpful discussions.

\end{document}